\documentclass[11pt,reqno]{amsart}
\usepackage{amsfonts,amssymb,amsmath}

\usepackage{ytableau}
\usepackage{nicematrix}
\usepackage{tikz}
\usepackage{xcolor}
\usepackage{enumerate}
\usepackage{multirow}
\usepackage{float}
\usepackage{mathtools}
\usepackage[colorlinks=true,citecolor=cyan,backref,pagebackref,
urlcolor=blue,linkcolor=blue]{hyperref}

\definecolor{fgreen}{RGB}{44,144, 14}

\renewenvironment{proof}{{\bfseries Proof.}}{\qed}

\numberwithin{equation}{section} 
\newtheorem{theorem}{Theorem}[section] 
\newtheorem{proposition}[theorem]{Proposition} 
\newtheorem{corollary}[theorem]{Corollary} 
\newtheorem{lemma}[theorem]{Lemma} 

\theoremstyle{definition}
\newtheorem{definition}[theorem]{Definition} 
\newtheorem{remark}[theorem]{Remark} 
\newtheorem{example}[theorem]{Example}

\everymath{\displaystyle}

\def\R{\mathbb R}

\def\C{\mathbb C}

\def\Z{\mathbb Z}

\def\R{\mathbb R}

\newcommand{\SL}{\mathrm{SL}}
\newcommand{\st}{\mathrm{Stab}}

\def\P{\mathbb P}

\def\SL{{\rm SL}}

\newcommand{\secref}[1]{Section~\ref{#1}}
\newcommand{\thmref}[1]{Theorem~\ref{#1}}
\newcommand{\lemref}[1]{Lemma~\ref{#1}}
\newcommand{\remref}[1]{Remark~\ref{#1}}
\newcommand{\propref}[1]{Proposition~\ref{#1}}
\newcommand{\corref}[1]{Corollary~\ref{#1}}

\makeatletter
\@namedef{subjclassname@2020}{\textup{2020} Mathematics Subject Classification}
\makeatother

\begin{document}

\title[On Character Variety of Anosov Representations]{On Character Variety of Anosov Representations}
\author[K. Gongopadhyay]{Krishnendu Gongopadhyay}
\author[T. Nayak]{Tathagata Nayak}

\address{Indian Institute of Science Education and Research (IISER) Mohali,
	Knowledge City,  Sector 81, S.A.S. Nagar 140306, Punjab, India}
\email{krishnendu@iisermohali.ac.in}

\address{Indian Institute of Science Education and Research (IISER) Mohali,
	Knowledge City,  Sector 81, S.A.S. Nagar 140306, Punjab, India}
 \email{tathagatanayak68@gmail.com}


\begin{abstract}
   Let $\Gamma$ be the fundamental group of a  $k$-punctured, $k \geq 0$,  closed connected orientable surface of genus $g \geq 2$. We show that the character variety of the $(Q^+, Q^-)$-Anosov irreducible representations, resp.  the character variety of the  $(P^+, P^-)$-Anosov Zariski dense representations of  $\Gamma$ into $\SL(n , \C)$, $n \geq 2$,  is a  complex  manifold of complex dimension \hbox{$(2g+k-2)(n^2-1)$}.  For $\Gamma=\pi_1(\Sigma_g)$, we also show that these character varieties are  holomorphic symplectic manifolds.
   \end{abstract}
   \textit{}

	
	\maketitle 
\section{Introduction}

In \cite{La}, Labourie introduced the notion of Anosov representations for surface groups into ${\rm PSL}(n,\R)$ during  the study of the Hitchin components. These representations have been an object of intense investigations since then, and have received many applications in the study of higher Teichm\"uller theories, see the surveys \cite{ka0}, \cite{ka}, \cite{kl}, \cite{os} \cite{po}, \cite{wi}. Guichard and Wienhard extended the notion of Anosov representations to representations of any word hyperbolic group $\Gamma$ into a semisimple Lie group $G$ in \cite{gw}. In particular, one can take $\Gamma$ as the group $F_k$, the free group of $k$ generators,  or, the fundamental group   $\pi_1(\Sigma_g)$ of a closed connected  orientable surface $\Sigma_g$ of genus $g \geq 2$.

 Now, suppose $G$ is a reductive algebraic group over $\C$ and let $\Gamma$ is a finitely generated group. Define $Hom(\Gamma,G)$ to be the set of homomorphisms from $\Gamma$ to $G$. The group $G$ acts on $Hom(\Gamma,G)$ by conjugation, and the categorical (GIT) quotient ${\mathcal X}=Hom(\Gamma,G)//G$ is called a \emph{character variety}.  In \cite{si}, Sikora investigated the geometry of $\mathcal X$ with emphasis on the case when $\Gamma=F_k$, or, $\pi_1(\Sigma_g)$. Sikora also discussed  reducibility  for subgroups of $G$. Recall that a closed subgroup $P$ of $G$ is called \emph{parabolic} if it contains a \emph{Borel subgroup} (a maximal connected solvable subgroup of $G$). A subgroup $H$ of $G$ is \emph{irreducible} if $H$ is not contained in any proper parabolic subgroup of $G$. The subgroup $H$ is called \emph{completely reducible} if for every parabolic subgroup $P \subset G$ containing $H$, there is some  Levi subgroup $L \subset P$ containing $H$. Obviously, every irreducible subgroup is completely reducible. It follows from  \cite{si} that each element of $\mathcal X$ corresponds to a unique conjugation orbit of completely reducible representations. Therefore,  for completely reducible representations, in particular for irreducible representations, the GIT orbits  and conjugation orbits may be identified, and for this reason we will not distinguish between the two. 

 The aim of this paper is to observe some results connecting the above two notions for $\Gamma=F_k$ or $\pi_1(\Sigma_g)$,  and $G=\SL(n, \C)$, $n, k, g \geq 2$. Note that $\SL(n,\C)$ is a Lie group, as well as semisimple algebraic group over $\C$. Thus both the above notions exist for $\SL(n, \C)$.  Recall that a pair of parabolic subgroups of $G$ is said to be \emph{opposite}  if their intersection is a Levi subgroup of both of them. Let $\C^n = D \oplus H$ be a decomposition of $\C^n$ into an $1$-dimensional subspace $D$ and  a $n-1$ dimensional subspace $H$. Denote, $Q^+ =\st (D)$ and $Q^-=\st (H)$. 
Then $(Q^+,Q^-)$ is a pair of opposite parabolic subgroups of $\SL(n,\C)$. Let  $\mathcal F^+=G/Q^+$ and  $\mathcal F^-= G/Q^-$.  By looking at the orbits of $D$ and  $H$ of the transitive group action of $\SL(n,\C)$ on $1$-dimensional and  $n-1$-dimensional subspaces respectively, we can identify,   $\mathcal F^+\cong \P(\C^n)$ and  $ \mathcal F^- \cong \P({\C^{n}}^{\ast})$. 

In \secref{3}, we prove \lemref{ml}, which shows that the the conjugation action of $\SL(n, \C)$ on the irreducible, resp.  Zariski dense,  Anosov representations  of $\Gamma$  into $\SL(n, \C)$ is closed and hence it gives us well-defined character variety of the respective representations, see   \propref{1.5} and \propref{1.6}. Now one can ask about the structure of these character varieties, and  we prove the following in this regard. 

\begin{theorem}
   Let $\Gamma=F_k$, resp.    $\pi_1(\Sigma_g)$, $k, g \geq 2$. The character variety of the $(Q^+, Q^-)$-Anosov irreducible representations, as well as the character variety of the  $(P^+, P^-)$-Anosov Zariski dense representations of $\Gamma$ into $\SL(n , \C)$, $n \geq 2$,  are complex  manifolds, each of complex dimension \hbox{$(k-1)(n^2-1)$}, resp. $(2g-2)(n^2-1)$.  
\end{theorem}

When $\Gamma=\pi_1(\Sigma_g)$, we also observe existence of a holomorphic symplectic form on these  Anosov character varieties. In view of this,  the above theorem may be rephrased for the surface groups as follows. 

\begin{corollary}
    
Let $\Gamma$ be the fundamental group of a $k$-punctured, $k \geq 0$,  closed connected orientable surface of genus $g \geq 2$.  Then the character variety of the $(Q^+, Q^-)$-Anosov irreducible representations, resp.  the character variety of the  $(P^+, P^-)$-Anosov Zariski dense representations of  $\Gamma$ into $\SL(n , \C)$, $n \geq 2$,  is a  complex  manifold of complex dimension \hbox{$(2g+k-2)(n^2-1)$}.  For $\Gamma=\pi_1(\Sigma_g)$, the respective character varieties are  holomorphic symplectic manifolds. 
\end{corollary}

The above results follow from \thmref{2.1}, \thmref{2.2} and \thmref{2.4} proven in this paper. We prove these results essentially by using methods developed in \cite{si} combining it with some results from \cite{gw}.  The corollary follows noting that the fundamental group of a geneus $g$ closed surface with $k$ punctures, $k\geq1$, is a free group of rank $2g+k-1$. 

Now we briefly mention the organization of the paper. We recall preliminary materials in \secref{prel}. In \secref{3}, we explore character varieties of Anosov representations. First we observe \lemref{ml} to have the character varieties well-defined, and then prove \thmref{2.1} and \thmref{2.2}. In \secref{s}, we observe the symplectic structure on the character variety  and then establish \thmref{2.4}.

In the following we give a summary of notations used in this paper for an affine complex connected semisimple algebraic group $G$, 

\noindent $Hom^i(\Gamma,G)=$ the set of irreducible representations,\\
$Hom^{zd}(\Gamma,G)=$ the set of Zariski dense representations,\\
$Hom_{(P^+,P^-)} (\Gamma, G)=$ the set of $(P^+, P^-)$-Anosov representations,\\
$Hom_{(P^+,P^-)}^{zd} (\Gamma, G)=$ the set of $(P^+, P^-)$-Anosov Zariski dense representations,\\ 
$X^i_G(\Gamma)=$ the character variety of irreducible representations of $\Gamma$ into $G$. \\
$X^{zd,A}_G(\Gamma)=$ the character variety of $(P^+, P^-)$-Anosov Zariski dense representations of $\Gamma$ into $G$.\\
$X^{i,A}_G(\Gamma)=$ the character variety of $(Q^+, Q^-)$-Anosov irreducible representations of $\Gamma$ into $G$ in the case of $G=\SL(n , \C)$. 

\section{Preliminaries} \label{prel}

\subsection{Anosov Representations} Let $G$ be a semisimple Lie group and $(P^{+},P^{-})$ be a pair of opposite parabolic subgroups of $G$ (as defined in \cite[Section 3.2]{gw}). Set, $\mathcal F^+=G/P^+$, resp. $ \mathcal F^-= G/P^-$. The subgroup $L=P^+ \cap  P^- $ is the Levi subgroup of both $P^+$ and $P^-$. The homogeneous space $\chi= G/L$ can be embedded into $G/P^+\times G/P^-$  in   the following way: 

Define a group action of $G$ on $G/P^+\times G/P^-$ such that $g.(\bar{x},\bar{y})=(\bar{gx},\bar{gy}) \; \forall\, g \in G,\bar{x} \in G/P^+, \bar{y} \in G/P^- $. It is easy to check that $\st\,(id_{G/P^+},id_{G/P^-})=P^+ \cap  P^-=L$. Therefore by the orbit-stabilizer theorem $G/L=G\,(id_{G/P^+},id_{G/P^-}) $. That is,  the homogeneous space  $\chi= G/L \subset G/P^+\times G/P^-$. Note that $\chi$ is also the unique open $G$-orbit under this action \cite[Section 2.1]{gw}.

Now take $\Gamma$ to be a finitely generated word hyperbolic group. Let $\partial_{\infty} \Gamma$ be the  boundary at $\infty$. Set, $\partial_{\infty} \Gamma^{(2)} =\partial_{\infty} \Gamma \times\partial_{\infty} \Gamma \smallsetminus \{{(t,t)~|~ t\in \partial_{\infty} \Gamma}\} $. Also note that a pair of points $(x_+,x_-)\in\mathcal{F}^+\times \mathcal{F}^-  $ is said to be \emph{transverse} if $(x_+,x_-)\in \chi \subset \mathcal{F}^+\times \mathcal{F}^-  $.
\begin{theorem}\cite[Theorem 2.9]{gw} Let $\Gamma$ be a finitely generated word hyperbolic group. Then there exists a proper hyperbolic metric space $\hat{\Gamma}$ such that 

$(1)$  $ \Gamma \times \mathbb{R}  \rtimes {\mathbb{Z} / 2\mathbb{Z}}$ acts on $\hat{\Gamma}$.

$(2)$ The $\Gamma \times {\mathbb{Z}/2\mathbb{Z}} $ action is isometric.

    $(3)$ Every orbit $\Gamma \to \hat{\Gamma}$ is a quasi-isometry. In particular, $ \partial_{\infty} \hat{\Gamma} \cong \partial_{\infty} \Gamma $.

    $(4)$ The $\mathbb R$ -action is free and every orbit $\mathbb{R} \to \hat{\Gamma}$ is a quasi-isometric embedding. The induced map ${\hat{\Gamma } / \mathbb R}\to \partial_{\infty}{\hat{ \Gamma}}^{(2)} =\partial_{\infty} \hat{\Gamma} \times\partial_{\infty} \hat{\Gamma} \smallsetminus \{(\hat{t},\hat{t})~|~ \hat{t}\in \partial_{\infty} \hat{\Gamma}\} $ is a homeomorphism.
\end{theorem}
In fact $\hat{\Gamma}$ is unique up to a $ \Gamma \times {\mathbb{Z} / 2\mathbb{Z}}$ -equivariant quasi-isometry sending $\mathbb R$-orbits to $\mathbb R$-orbits. Denote by $\phi_t$ the $\mathbb R$-action on $\hat{\Gamma}$ and by $(\tau ^+,\tau^-):\hat{\Gamma} \to {\hat{\Gamma } / \mathbb R} \cong \partial_{\infty}{\hat{ \Gamma}}^{(2)} \cong \partial_{\infty} \Gamma^{(2)}$ the maps associating to a point the endpoints of its $\mathbb R$-orbit.
\begin{definition}\cite[Definition 2.10]{gw}  
A representation $\rho: \Gamma \to G $ is said to be  $(P^+,P^-)$-Anosov if there exists continuous $\rho$-equivariant maps $\xi^{+}, \hbox{resp. } \xi^{-}: \partial_{\infty} \Gamma \to \mathcal F^+, \hbox{ resp.  } \mathcal F^{-}$ such that:
\begin{enumerate}
\item  $\forall (t^+,t^-) \in \partial_{\infty} \Gamma^{(2)} ,\: (\xi^+(t^+),\xi^-(t^-))$ is transverse.

\item For one (and hence any) continuous and equivariant family of norms ${(\,{\|\,.\,\|}_{\hat{m}}\,)}_{{\hat{m}}\in \hat{\Gamma}}$ on $(\,T_{\xi^+({\tau ^+ }(\hat{m}))} \mathcal{F}^+)_{\hat{m} \in \hat{\Gamma}}$ (\,\hbox{resp.}\,$(\,T_ {\xi ^-(\tau ^-(\hat{m}))} \mathcal{F} ^-)_{\hat{m} \in \hat{\Gamma}}$\,), there exists $ A,a>0  $  such that $\forall t \geq 0,\,\hat{m}\in\hat{\Gamma}$ \,and \, $ e\in (\, T_{\xi^+(\tau^+(\hat{m}))} \mathcal{F}^+)_{\hat{m} \in \hat{\Gamma} }$ (\,\hbox{resp.} $e \in (\,T_ {\xi ^-{(\tau ^-(\hat{m}))}} \mathcal{F} ^-)_{\hat{m} \in \hat{\Gamma}}$\,):\\ $ {\|e\|}_{\phi_{-t} \hspace{0.06cm} (\hat{m})} \leq Ae^{-at} {\|e\|}_{\hat{m}}$ ( \hbox{resp.}$ {\|e\|}_{\phi _t \hspace{0.06cm} (\hat{m})} \leq Ae^{-at} {\|e\|}_{\hat{m}}$ ) .
 \end{enumerate}
 The maps $\xi^{\pm}$ are called its \emph{Anosov maps}.
\end{definition}

Using \cite[Definition 4.8,\,4.9]{gw}, it can be concluded that  $(\xi ^+, \xi ^-)$ is said to be compatible if 
\begin{enumerate}
    \item $\forall t \in \partial_{\infty} \Gamma,\: \st(\xi^+(t)) \cap \st(\xi^-(t)) $ is a parabolic subgroup of $G$ and 

\item $\forall t^+ \neq t^- \in \partial_{\infty} \Gamma, (\xi^+(t^+), \xi^-(t^-) ) \in \chi= 
G/{(P^+\cap P^-)} $.
\end{enumerate}

In \cite{gw}, Guichard and Weinhard showed that under certain conditions, only the  existence of equivariant continuous compatible maps $(\xi ^+, \xi ^-)$ implies the Anosov property of a representation. This is described in details in \propref{1.3} and \propref{1.4}. 

\begin{proposition}\label{1.3}
    \cite[Prop. 4.10]{gw}
    Let $\rho: \Gamma \to \SL(n, \C)$ be a representation. Suppose that
\begin{enumerate}
   \item $\rho$ is irreducible.

   \item $\rho$ admits a compatible pair $\xi ^\pm : \partial_{\infty} \Gamma \to  {\SL(n, \C)/{Q^\pm}}$ of continuous $\rho$-equivariant maps
\end{enumerate}
   Then $\rho$ is $(Q^+,Q^-)$-Anosov and $\xi^\pm$ are its Anosov maps.
\end{proposition}

\begin{proposition}\label{1.4}
  \cite[Theorem 4.11]{gw}  
  Let $\rho: \Gamma \to G $ be a Zariski dense representation where $G$ is an affine complex connected semisimple algebraic group and $(P^+,P^-)$ is a pair of opposite parabolic subgroups of $G$. Suppose that $\rho$ admits a pair of equivariant continuous compatible maps $\xi^{+}, \hbox{resp. } \xi^{-}: \partial_{\infty} \Gamma \to  G/{P^+}, \hbox{ resp.  } G/{P^-} $. Then the representation $\rho$ is $(P^+,P^-)$-Anosov and $\xi^\pm$ are its Anosov maps.
\end{proposition}

\begin{theorem}\label{1.7}
    \cite[Theorem 5.13]{gw} The set of $(P^+,P^-)$-Anosov representations of $\Gamma$ into $G$, denoted by $Hom_{(P^+,P^-)} (\Gamma, G)$,  is open in $Hom (\Gamma, G)$.
    \end{theorem}
    \begin{remark}\label{1.8}
    Since $\Gamma=F_k$ or  $\pi_1(\Sigma_g)$ is finitely generated, in $Hom(\Gamma,G)$,  the compact-open topology coincides with  the topology of pointwise convergence, cf. \cite[Section 4.2.4]{ka}. We will call this topology of pointwise convergence as      complex topology. Therefore, in the above theorem, these two topologies are same.
         \end{remark}

\subsection{Varieties of Representations}
Let $\Gamma$ be a finitely generated group and $G$ be an affine complex algebraic group, then $Hom(\Gamma,G)$ is an affine algebraic set. In our case, $\Gamma = F_k$ or $\pi_1(\Sigma_g) $ is finitely generated and $G$ is an affine complex connected semisimple algebraic group. Hence $Hom(\Gamma,G)$ is an affine variety (here affine variety means affine algebraic set irrespective of irreducibility).

\begin{example}
    Take $\Gamma=\pi_1(\Sigma_2)=\langle a_1,b_1,a_2,b_2 ~ | ~ [a_1,b_1][a_2,b_2]=1 \rangle, G=\SL(2,\C)$. Each point $\rho: \Gamma \to \SL(2,\C)$ of $Hom(\Gamma,\SL(2,\C))$ is represented by;
    
$\rho(a_1)=\begin{bmatrix}
x_1 & x_2 \\
x_3 & x_4
\end{bmatrix}, \rho(b_1)=\begin{bmatrix}
x_5 & x_6 \\
x_7 & x_8
\end{bmatrix}, \rho (a_2)=\begin{bmatrix}
x_9 & x_{10} \\
x_{11} & x_{12}
\end{bmatrix}, \rho (b_2)=\begin{bmatrix}
x_{13} & x_{14 }\\
x_{15} & x_{16}  
\end{bmatrix}$
 such that $$x_1x_4-x_2x_3=x_5x_8-x_6x_7=x_9x_{12}-x_{10}x_{11}=
 x_{13}x_{16}-x_{14}x_{15}=1 $$  and  
$$ \rho(a_1)\,\rho(b_1)\,{{\rho(a_1)}^{-1}}\,{{\rho(b_1)}^{-1}}\,\rho(a_2)\,\rho(b_2)\,{{\rho(a_2)}^{-1}}\,{{\rho(b_2)}^{-1}}=\begin{bmatrix}
1 & 0\\
0 & 1  
\end{bmatrix}.$$
Therefore, $Hom(\pi_1(\Sigma_2), \SL(2,\C))\subset \C^{16}$ is an affine algebraic set.

If $\Gamma=F_k $ with free generators say $a_1,\ldots,a_k$, then the map $\rho \to (\rho(a_1), \ldots,\rho(a_k))$ for any $\rho \in Hom(F_k,G) $ yields an isomorphism $Hom(F_k,G) \cong G^k$.
\end{example}

\begin{proposition}\label{1.10}
    \cite[Prop. 27]{si} The set of all irreducible representations of $\Gamma$ in $G$, denoted by $Hom^i (\Gamma, G)$,  is Zariski open in $Hom(\Gamma, G)$.
\end{proposition}

\begin{proposition}\label{1.11}
\cite[Prop. 8.2]{ab}
    The set of all Zariski dense representations of $\Gamma$ in $G$ where $G$ is connected and reductive, denoted by $Hom^{zd} (\Gamma, G)$, is Zariski open in $Hom(\Gamma, G)$. 
\end{proposition}
In our case, $G$ is an affine complex connected  semisimple algebraic group and its semisimplicity implies it is reductive.
\begin{remark}\label{1.12}
Since the complex topology is finer than the Zariski topology,
     $Hom^i(\Gamma, G)$ and $Hom^{zd} (\Gamma, G)$ are open in $Hom(\Gamma, G)$ in Zariski topology  implies that they are also open in the  complex topology in $Hom(\Gamma, G)$.
\end{remark}

\begin{corollary}\label{1.13}
 Let $\Gamma=F_k$ or $\pi_1(\Sigma_g)$ where $  g\geq2$. Then the set of  $(Q^+, Q^-)$-Anosov irreducible representations of $\Gamma$ into $\SL(n , \C)$, denoted by   ${Hom_{(Q^+,Q^-)}^i (\Gamma, \SL(n,\C))}$,  is open in $Hom (\Gamma,\SL(n,\C))$,  as well as in $Hom^i (\Gamma,\SL(n,\C))$ in the complex topology.

\end{corollary}
\begin{proof}
    Note that, $${Hom_{(Q^+,Q^-)}^i (\Gamma,\SL(n,\C))}={{Hom_{(Q^+,Q^-)} (\Gamma,\SL(n,\C))} \cap {Hom^i (\Gamma,\SL(n,\C))}}.$$
    Since intersection of two open sets are open, using \thmref{1.7} and \remref{1.12}, we get our conclusion.
\end{proof}

Now we will associate an affine algebraic scheme with  $Hom(\Gamma,G)$  so that the set of closed points of this affine scheme coincides with $Hom (\Gamma,G)$. The following paragraph follows from \cite[Section 5 ]{si}. We will only mention the important results here.

Suppose $A$ is any commutative $\C$-algebra. Let $G(A)$ denotes the set of $\C$-algebra homomorphisms from $\C[G]$ to $A$. This is equipped with a group structure defined in \cite[Section 5 ]{si}. For example, $G(A)= \SL(n,A)$ for $G=\SL(n,\C).$ A commutative $\C$-algebra $R(\Gamma,G)$ is said to be a universal representation algebra of $\Gamma$ into $G$ and $\rho_U : \Gamma \to G(R(\Gamma,G))$ is a universal representation if for every commutative $\C$-algebra $A$ and every representation $\rho : \Gamma \to G(A)$, there exists a $\C$-algebra homomorphism $f: R(\Gamma,G) \to A$ inducing a representation $G(f):G(R(\Gamma,G)) \to G(A)$ such that $\rho = G(f)\circ \rho_U$. \cite[Lemma 26]{si} assures us that $R(\Gamma, G)$ and $\rho_U$ always exist, $R(\Gamma, G)$ is well defined up to an isomorphism of $\C$-algebras and $\rho_U : \Gamma \to G(R(\Gamma,G))$ is unique up to a composition with $G(f)$ where $f$ is a $\C$-algebra automorphism of $R(\Gamma,G).$ Denote $\mathcal{H}om (\Gamma, G)= Spec \:  R(\Gamma,G)$. By taking $A=\C$, it can be shown that there is a one to one correspondence between the points of $Hom(\Gamma,G)$ and the closed points of $\mathcal{H}om (\Gamma, G)$ i.e.
$$\{Hom(\Gamma,G)\}\xleftrightarrow {1-1}\{Closed \hspace{0.08cm} points \hspace{0.08cm} of \hspace{0.08cm}\mathcal{H}om (\Gamma, G)\} = \{MSpec \: R(\Gamma,G) \}.$$
Also, $R(\Gamma,G)/\sqrt{0} = \C[Hom(\Gamma,G)]$-coordinate ring of $Hom(\Gamma,G)$. From this relation, it can be said that the Zariski topology on $\mathcal{H}om (\Gamma, G)= Spec \hspace{0.08cm}   R(\Gamma,G)$ is homeomorphic to the Zariski topology on $Spec \hspace{0.08cm} \C [Hom(\Gamma,G)]$, see  \cite[page 23]{am}.

Note that, a point $x$ of an affine variety or an affine scheme $X$ is said to be simple if $x$ belongs to a unique irreducible component of $X$ and the dimension of that component is equal to the dimension of the Zariski tangent space to $X$ at $x$ (denoted by $T_x X$). We say that $\rho: \Gamma \to G$ is smooth if $\rho$ is  a simple point of the affine variety $Hom(\Gamma,G)$ and $\rho$ is said to be scheme smooth if the corresponding maximal ideal of $\rho$  in $\mathcal{H}om (\Gamma, G)$ (using the above $1-1$ correspondence) is a simple point in $\mathcal{H}om (\Gamma, G)$. If $\rho$ is scheme smooth, then it is smooth also. \cite[Section 9]{si}

Note that, $Hom(F_k,G)=G^ k$ (Cartesian product of $k$-copies of $G$) is a non-singular affine variety. Therefore every $\rho \in Hom(F_k,G) $ is smooth. By Proposition\,37 of \cite{si}, all irreducible representations of $\pi_1(\Sigma_g)$ in a reductive group $G$ are scheme smooth, hence smooth. From these facts, we have the following.

\begin{proposition} {\rm \cite{si}} \label{1.16}
    For $\Gamma=F_k$, resp. $\pi_1(\Sigma_g)$,  and every reductive group $G$, all irreducible representations $\rho:\Gamma \to G$ are smooth.
\end{proposition}

\section{Character Varieties of Anosov representations}\label{3}

In this section we will assume $G$ to be an affine complex connected semisimple algebraic group. In particular,  $\SL(n, \C)$ is also an affine complex connected semisimple algebraic group and any affine complex algebraic group is a Lie group.

In \cite[Remark 4.12]{gw}, it is shown that for a pair of opposite parabolic subgroups $(P^+, P^-)$ of $G$, an irreducible $G$-module $V$ can always be constructed such that $V$ admits a decomposition into a line $D$ and a hyperplane $H$ so that $P^+=\st_G(D)$ and $P^-=\st_G(H)$.

\begin{lemma}\label{ml} 
Let $\rho: \Gamma \to G$ be  a $(P^+, P^-)$-Anosov representation, and let $(\xi^+, \xi^-)$ be the associated Anosov maps. Then $\xi'^{+}, \hbox{resp. } \xi'^{-}: \partial_{\infty} \Gamma \to \mathcal F^+ =G/{P^+}, \hbox{ resp.  } \mathcal F^-  =G/{P^-}$,  defined for all $t \in \partial_{\infty} \Gamma$, 
$\xi'^{\pm}(t)=g \xi^{\pm }(t)$ is a compatible pair of $g \rho g^{-1}$-equivariant continuous maps, where $g \rho g^{-1}$ is given by
$$g \rho g^{-1}(v)=g \rho(v) g^{-1}, \hbox{ for all } v \in \Gamma, ~ g \in G. $$
 \end{lemma}

 \begin{proof}
We have $\xi^{\pm}: \partial_{\infty} \Gamma \to \mathcal F^{\pm}$ is the  compatible pair of  $\rho$-equivariant continuous maps by the definition of Anosov representations. This implies the following.

\subsubsection*{ $\xi'^{\pm}$ are $g \rho g^{-1}$-equivariant}   Note that for $v \in \Gamma$, $t \in \partial_{\infty} \Gamma$, 
$$\xi'^{\pm}(vt)=g\xi^{\pm}(vt)=g \rho(v) \xi^{\pm }(t)=g \rho(v) g^{-1} g \xi^{\pm}(t)=g \rho(v)g^{-1} \xi'^{\pm}(t). $$

\subsubsection*{ The maps ${\xi'}^{\pm}$ are continuous} Since $\xi^{\pm}: \partial_{\infty} \Gamma \to \mathcal F^{\pm}$ are continuous, $\xi'^{\pm}=g \xi^{\pm}$ are also continuous, since the action of $G$ on $\mathcal F^{\pm}$ are continuous. 

\subsubsection*{Compatibility of $(\xi'^+, \xi'^-)$} 
Suppose $\st(\xi^+(t)) \cap \st(\xi^-(t))=Q$, where $Q$ is a parabolic subgroup of $G$. It is easy to see that $\st(g \xi^{\pm}(t))=g \st(\xi^{\pm}(t))g^{-1}$.  Thus,\begin{eqnarray*}
\st(\xi'^+(t)) \cap \st(\xi'^-(t))&=& g \st(\xi^+(t)) g^{-1} \cap g \st(\xi^-(t)) g^{-1}\\ & =&  g (\st(\xi^+(t) \cap \st(\xi^-(t)) g^{-1}=g Q g^{-1}. \end{eqnarray*}
which is also a parabolic subgroup of $G$.

\subsubsection*{Transversility of $(\xi'^+(t^+), \xi'^-(t^-)$ }
    Note that  
    $(\xi'^+(t^+), \xi'^-(t^-))=(g\xi^+(t^+)\,,\, g\xi^-(t^-))$. \hbox{Using} transversility of $(\xi^+, \xi^-)$, we have $(\xi^+(t^+), \xi^-(t^-)) \in \chi \subset \mathcal F^+ \times \mathcal F^-$. Since $\chi$ is a homogeneous space, in particular , an unique open $G$ orbit in  
    $\mathcal F^+ \times \mathcal F^-$ under diagonal action of $G$, hence $g(\xi^+(t^+), \xi^-(t^-))$$=(g\xi^+(t^+)\,,\, g\xi^-(t^-))$ belongs to $\chi$. Thus $(\xi'^+(t^+), \xi'^-(t^-)) \in \chi$ for all $t^+, ~t^- \in \partial_{\infty}\Gamma$, $ t^+ \neq t^-$. 

This proves the lemma. 
\end{proof}

Now using \propref{1.3}, \lemref{ml} and the fact that conjugate of an irreducible subgroup is also irreducible, we obtain the following proposition.

\begin{proposition}\label{1.5}
   Let $\rho: \Gamma \to \SL(n, \C)$ be an irreducible  $(Q^+, Q^-)$-Anosov representation with associated Anosov maps $(\xi^+, \xi^-)$. Then $ g\rho {g^{-1}}: \Gamma \to \SL(n, \C)$ be an irreducible  $(Q^+, Q^-)$-Anosov representation with associated Anosov maps $(g\xi^+, g\xi^-)$ $\forall g \in \SL(n, \C)$.
\end{proposition}

Similarly using \propref{1.4}, \lemref{ml} and the fact that conjugate of a Zariski dense subgroup is also Zariski dense, the following proposition can be obtained.
\begin{proposition}\label{1.6}
   Let $\rho: \Gamma \to G$ be a Zariski dense  $(P^+, P^-)$-Anosov representation with associated Anosov maps $(\xi^+, \xi^-)$. Then $ g\rho {g^{-1}}: \Gamma \to G$ be a Zariski dense  $(P^+, P^-)$-Anosov representation with associated Anosov maps $(g\xi^+, g\xi^-)$ $\forall g \in G$. 
\end{proposition}

\subsection{Irreducible and Zariski dense representations}    
From \propref{1.5},  it follows that $\SL(n,\C)$ acts on ${Hom_{(Q^+,Q^-)}^i (\Gamma,\SL(n,\C))}$  by conjugation. Similarly, it follows from \propref{1.6}   
 that $G$ acts on the set of Zariski dense $(P^+,P^-)$-Anosov representations (denoted by $Hom_{(P^+,P^-)}^{zd} (\Gamma, G)$) by conjugation.  

This action gives a quotient space structure on $Hom_{(Q^+,Q^-)}^i (\Gamma, G)$ for $G=\SL(n,\C)$ which is  denoted by ${Hom_{(Q^+,Q^-)}^i (\Gamma, G)} / G$  where any two $\rho, \rho' \in Hom_{(Q^+,Q^-)}^i (\Gamma, G)$ are equivalent  if there exists $g \in G$ such that $\rho'=g\rho g^{-1}$. Similarly we get a quotient space structure on $Hom_{(P^+,P^-)}^{zd}(\Gamma, G)$ for any affine complex connected semisimple  algebraic group $G.$

\begin{proposition}\label{1.14}
The Zariski dense representations are irreducible, that is, 
$Hom^{zd} (\Gamma, G)\subset Hom^i(\Gamma, G)$. 
\end{proposition}
\begin{proof}
    Suppose $\rho: \Gamma \to G$ is a Zariski dense represesentation. So, the Zariski closure of $\rho (\Gamma)$ is $G$. By definition, the parabolic subgroups of  $G$ are Zariski closed. Therefore, $\rho (\Gamma)$ can not be contained in any proper parabolic subgroup of $G$. Note that a subgroup is called irreducible if it is not contained in any proper parabolic subgroup. This implies $\rho (\Gamma)$ is an irreducible subgroup of $G$.
\end{proof}

\begin{corollary}\label{1.15}
   ${{Hom^{zd}}_{(P^+,P^-)}(\Gamma, G)}$ is open in $Hom (\Gamma,G)$, as well as in $Hom^i (\Gamma,G)$ in the  complex topology. 
\end{corollary}
\begin{proof}
    Note that, $Hom_{(P^+,P^-)}^{zd} (\Gamma,G)$ is the intersection of two open sets   $Hom_{(P^+,P^-)} (\Gamma,G)$ and $Hom^{zd} (\Gamma,G)$   by \thmref{1.7} and \remref{1.12}. Therefore it is open in $Hom(\Gamma,G)$. Now from \propref{1.14}, it follows that $Hom_{(P^+,P^-)}^{zd} (\Gamma,G)$  is also open in $Hom^i (\Gamma,G)$. \end{proof}

\subsection{Character Variety}
We know that $G$ acts on $Hom^i(\Gamma,G)$ via conjugation. Denote $Hom^i(\Gamma,G)/G=X^i_G(\Gamma)$,   and \; ${{Hom_{(P^+,P^-)}^{zd}(\Gamma, G)}/G}={X^{zd,A}_G(\Gamma)}$. When $G=\SL(n, \C)$, we denote $Hom_{(Q^+,Q^-)}^i (\Gamma, \SL(n,\C))/\SL(n,\C)=X^{i,A}_{\SL(n,\C)}(\Gamma)$. 

\medskip Now, observe that all of these quotient spaces consist of some orbits of irreducible (which implies completely reducible) representations of $Hom(\Gamma,G)$. Since each element of the categorical quotient is represented by a unique conjugation orbit of completely reducible representations \cite[Section\,11]{si}, in our cases the categorical quotients are same as the quotient spaces which are obtained by conjugation $G$-actions. Therefore, we will call these quotient spaces as character varieties of respective types of representations of $\Gamma$ in $G$. Note that, these quotient spaces which we are calling as character varieties may not be algebraic varieties. In particular, these are subsets of ${\mathcal X}=Hom(\Gamma,G)//G.$

On the other hand, it follows from \propref{1.10}  that $Hom^i(\Gamma,G)$ is an algebraic variety. From \propref{1.16}, it can be said that $Hom^i(\Gamma,G)$ is a non-singular algebraic variety. Therefore, $Hom^i(\Gamma,G)$ forms a complex manifold in the complex topology \cite[Chapter \,7, \,8]{shf}. Since   $Hom_{(P^+,P^-)}^{zd}(\Gamma, G)$,  is open in   $Hom^i(\Gamma,G)$,  in the complex topology,  by    \corref{1.15}, $Hom_{(P^+,P^-)}^{zd}(\Gamma, G)$\,) is also a complex manifold  (\,open submanifold of $Hom^i(\Gamma,G)$\,). 

By the same argument, using \corref{1.13}, it can be said that  $ Hom_{(Q^+,Q^-)}^i (\Gamma, \SL(n,\C))$ is an open complex submanifold of  $Hom^i(\Gamma,\SL(n,\C))$.

 It is said that a reductive group $G$ has property CI if the centralizer of every irreducible subgroup coincides with center of $G$. Note that,   $G=\SL(n,\C)$ has property CI \cite[Example 18]{si}. From here, the following follows. 
 
 \begin{corollary}\label{cor}
Let $G=\SL(n,\C)$. Then      $G/{C(G)}$ acts freely on $Hom^i(\Gamma,G)$, $Hom_{(Q^+,Q^-)}^i (\Gamma, G)$    and $Hom_{(P^+,P^-)}^{zd} (\Gamma, G)$  respectively.  
\end{corollary}

Also note that $$Hom^i(\Gamma,G)/{(G/C(G))}=Hom^i(\Gamma,G)/G=X^i_G(\Gamma),$$  
$$Hom_{(P^+,P^-)}^{zd} (\Gamma, G)/{(G/C(G))}= Hom_{(P^+,P^-)}^{zd} (\Gamma, G)/G=X^{zd,A}_G(\Gamma).$$  This is also true for $Hom_{(Q^+,Q^-)}^i(\Gamma, \SL(n,\C))$.

The following proposition follows using similar ideas as in the proof of \cite[Proposition\,49]{si}. 
\begin{proposition}\label{1.17}
Let $\Gamma=F_k$ or $\pi_1(\Sigma_g)$, $g \geq 2$, and  $G=\SL(n,\C)$. Then
    $X^i_G(\Gamma)$ and  $X^{i,A}_G(\Gamma)$ are complex manifolds. 
\end{proposition}
\begin{proof}
 $G/{C(G)}$ action on $Hom^i(\Gamma,G)$ is properly discontinuous by \cite[Proposition\,1.1]{jm}. Since $Hom_{(Q^+,Q^-)}^i (\Gamma, G)$ is open in $Hom^i (\Gamma, G)$, $G/{C(G)}$ action on $Hom_{(Q^+,Q^-)}^i (\Gamma, G)$ is also properly discontinuous. From \corref{cor},  $G/{C(G)}$ action is  free. The quotient of a complex manifold by a free and properly discontinuous group action is also a complex manifold. Therefore $X^i_G(\Gamma)$ and  $X^{i,A}_G(\Gamma)$ are complex manifolds.
    \end{proof}
 
\medskip Since $Hom_{(Q^+,Q^-)}^i (\Gamma, G)$ is open in $Hom^i (\Gamma, G)$ (here  $G=\SL(n,\C)$), their quotient spaces by conjugation  $G$-action follow the same relation, that is,   $X^{i,A}_G(\Gamma)$ is an open submanifold of $X^i_G(\Gamma)$. Therefore $\dim Hom_{(Q^+,Q^-)}^i (\Gamma, G)=\dim Hom^i (\Gamma, G)$ and $\dim X_G^{i,A} (\Gamma)=\dim X^i_G (\Gamma)$ as complex manifolds.

\begin{proposition}\label{1.18}
    For $\Gamma=F_k$, \hbox{resp.} $\pi_1(\Sigma_g)$ where $g \geq 2$, $\dim Hom^i (\Gamma, G)= k \dim  G$, resp. $(2g-1) \dim G$,  where $G$ is any affine complex connected semisimple algebraic group.
\end{proposition}
\begin{proof}
     Since $Hom^i (\Gamma, G)$ is Zariski open in $Hom (\Gamma, G)$, therefore $$\dim Hom^i (\Gamma, G)=\dim Hom (\Gamma, G).$$ Now consider the case for $\Gamma=\pi_1(\Sigma_g)$. By \cite[Lemma 1]{go83}, any connected component of $Hom (\pi_1(\Sigma_g), G)$ is of dimension $(2g-1)\dim G$.  So, $$\dim Hom (\pi_1(\Sigma_g), G)=\dim Hom^i (\pi_1(\Sigma_g), G)=(2g-1)\dim G.$$ When $\Gamma=F_k$, $$\dim Hom (F_k, G)=\dim G^ k= k  \dim G =\dim Hom^i(F_k, G).$$ This completes the proof.
     \end{proof}
\begin{remark}
    In the above proof, we have considered the respective spaces as  algebraic varieties while mentioning their dimensions.  We know that, any non-singular algebraic variety  admits a complex manifold structure and its    dimension as algebraic variety coincides with the dimension as complex manifold. We have already seen that, $Hom^i (\Gamma, G)$ is a non-singular algebraic variety. Therefore, both of these dimensions are same for $Hom^i (\Gamma, G)$.
\end{remark}

\begin{proposition}\label{1.19}
When $G=\SL(n,\C),$
\begin{enumerate}
    \item For $\Gamma=F_k$ where $k\geq 2$, $\dim X^i_G (F_k)=(k-1) \dim G=\dim X_G^{i,A} (F_k)$ as complex manifold.

    \item For $\Gamma=\pi_1(\Sigma_g)$ where $g\geq 2$, $\dim X^i_G (\pi_1(\Sigma_g))=(2g-2)\dim\hspace{0.08cm}G=\dim X_G^{i,A} (\pi_1(\Sigma_g))$ as complex manifold.
    \end{enumerate}
\end{proposition}
\begin{proof}
    Since, $G/{C(G)}$ acts freely on $Hom^i (\Gamma, G)$ for $G=\SL(n,\C)$, $$\dim X^i_G (\Gamma)=\dim Hom^i (\Gamma, G)-\dim G/{C(G)}=\dim Hom^i (\Gamma, G)-\dim G+\dim C(G).$$ In our case, $\dim C(G)=0$ since $G=\SL(n,\C)$ is semisimple. Now,  put the corresponding values of $\dim Hom^i (\Gamma, G)$ from \propref{1.18} for $\Gamma=F_k$, resp.  $\pi_1(\Sigma_g)$.
\end{proof}

  Using similar ideas as in the proof of \cite[Proposition\,49]{si}, the following proposition is obtained. 
\begin{proposition}\label{1.20}
Let $\Gamma=F_k$ or $\pi_1(\Sigma_g)$, $g \geq 2$ and  $G$ is any affine complex connected semisimple  algebraic group. Then 
    $X^i_G(\Gamma)$ and  $X^{zd,A}_G(\Gamma)$ are complex orbifolds. 
\end{proposition}
\begin{proof}
    $G/{C(G)}$ action on the complex manifold $Hom^i(\Gamma,G)$ is properly discontinuous by \cite[Proposition\,1.1]{jm}. Since $Hom_{(P^+,P^-)}^{zd} (\Gamma, G)$ is open in $Hom^i (\Gamma, G)$, $G/{C(G)}$ action on the complex manifold $Hom_{(P^+,P^-)}^{zd} (\Gamma, G)$ is also properly discontinuous. From \cite[Proposition\,15]{si}, the centralizer of an irreducible subgroup of  $G$ is a finite extension of $C(G)$. Therefore $G/{C(G)}$ acts on $Hom^i (\Gamma, G)$ and  $Hom_{(P^+,P^-)}^{zd} (\Gamma, G)$ with finite stabilizers. The quotient of a complex manifold by  properly discontinuous group action with finite stabilizers is  a complex orbifold. Therefore,  $X^i_G(\Gamma)$ and  $X^{zd,A}_G(\Gamma)$ are complex orbifolds.
\end{proof}

Observe that, as $Hom_{(P^+,P^-)}^{zd} (\Gamma, G)$ is open in $Hom^i (\Gamma, G)$,  hence  $X^{zd,A}_G(\Gamma)$ is  open in $X^i_G(\Gamma)$. Therefore $\dim 
 Hom_{(P^+,P^-)}^{zd} (\Gamma, G)=\dim Hom^i (\Gamma, G)$,  and $\dim  X_G^{zd,A} (\Gamma)=\dim X^i_G (\Gamma)$ as complex orbifolds.

\begin{proposition}\label{1.21}
    Let $G$ be any affine complex connected semisimple algebraic group.
    \begin{enumerate}
    \item For $\Gamma=F_k$ where $k\geq 2$, $\dim X^i_G (F_k)=(k-1)\dim G=\dim X_G^{zd,A} (F_k)$ as a complex orbifold.

    \item For $\Gamma=\pi_1(\Sigma_g)$ where $g\geq2$,  $\dim X^i_G (\pi_1(\Sigma_g))=(2g-2)\dim G=\dim X_G^{zd,A} (\pi_1(\Sigma_g))$ as a complex orbifold.
    \end{enumerate}
\end{proposition}

\begin{proof}
It follows from \cite[Proposition\,15]{si} that, 
  $G/{C(G)}$ acts  on $Hom^i (\Gamma, G)$  with finite stabilizers. Therefore, $$\dim X^i_G (\Gamma)=\dim Hom^i (\Gamma, G)-\dim G/{C(G)}=\dim Hom^i (\Gamma, G)-\dim G+\dim C(G).$$ In our case, $\dim C(G)=0$ since $G$ is semisimple.  Now using \propref{1.18},  the result is obtained.
\end{proof}

\begin{theorem}\label{2.1}
 
Let $\Gamma=F_k$, resp.  $\pi_1(\Sigma_g)$, $k,g \geq 2$. The character variety of $(Q^+, Q^-)$-Anosov irreducible representations of $\Gamma$ into $\SL(n , \C)$ is a complex manifold of complex  dimension \hbox{$(k-1)(n^2-1)$}, resp. $(2g-2)(n^2-1) $. 
\end{theorem}
\begin{proof}
        This proof follows from \propref{1.17} and \propref{1.19} and the fact that the complex dimension of $\SL(n, \C)$ is $n^2-1$.
\end{proof}

We also have the following theorem concerning Zariski dense Anosov representations.  
\begin{theorem}
        Let $\Gamma=F_k$, resp.  $\pi_1(\Sigma_g)$  where $k,g \geq 2$. Let $G$ be an affine complex connected semisimple algebraic group and $(P^{+},P^{-})$ be a pair of opposite parabolic subgroups of $G$. The character variety of $(P^+, P^-)$-Anosov Zariski dense representations of $\Gamma$ into $G$ is a complex orbifold of complex dimension $(k-1) \dim G$, resp. $(2g-2) \dim G$.
\end{theorem}
    
\begin{proof} When $G$ is an affine complex connected semisimple algebraic group, it follows from \propref{1.20} and \propref{1.21} that,  the character varieties of $(P^+, P^-)$-Anosov Zariski dense representations of $\Gamma$ into $G$ are complex orbifolds of respective dimensions. \end{proof}

As a special case of the above result, we obtain the following.

   \begin{theorem}\label{2.2}
    Let $\Gamma=F_k$, resp.  $\pi_1(\Sigma_g)$  where $k,g \geq 2$.  The character variety of $(P^+, P^-)$-Anosov Zariski dense representations of $\Gamma$ into $\SL(n, \C)$ is  a complex manifold of complex dimension $(k-1)(n^2-1)$, resp. $(2g-2)(n^2-1)$.
\end{theorem}

\begin{proof} Now, take $G=\SL(n,\C)$ in \propref{1.20}. Then $G/{C(G)}$ action on the complex manifold  $Hom_{(P^+,P^-)}^{zd} (\Gamma, G)$  will be free (\,since $\SL(n,\C)$ has property  CI\,). It follows from \propref{1.20}, this action is properly discontinuous. Hence, its character variety $X^{zd,A}_G(\Gamma)$ will become a complex manifold. It is obvious that $\dim X^i_G (\Gamma)=\dim 
 X_G^{zd,A} (\Gamma)$ as complex manifold. Now, its dimension can be concluded from \propref{1.19} and noting the fact that the complex dimension of $\SL(n, \C)$ is $n^2-1$. \end{proof}

\section{Symplecticity of Character Varieties of Anosov Representations }\label{s}

A representation $\rho: \Gamma \to G$ is said to be good if it is irreducible and the centralizer of $\rho(\Gamma)$ is center of $G$.  Now take $G=\SL(n,\C)$. Then  a representation is good if and only if  it is irreducible due to the CI property of $\SL(n,\C)$. Denote $[\rho]$ to be the equivalence class of $\rho \in Hom(\Gamma,G)$ under the 
 conjugation  $G$-action.

Since $X^{i,A}_G(\Gamma)$ and $X^{zd,A}_G(\Gamma)$ are open complex submanifolds of $X^i_G(\Gamma)$, by \thmref{2.1} and  \thmref{2.2}  respectively, we have $T_{[\rho]}\, {X^i_G(\Gamma)}=T_{[\rho]}\, {X^{i,A}_G(\Gamma)}$, and $ T_{[\rho]}\, {X^i_G(\Gamma)}=T_{[\rho]}\, {X^{zd,A}_G(\Gamma)}$  for all  $[\rho]$ in ${X^{i,A}_G(\Gamma)}$,resp.  ${X^{zd,A}_G(\Gamma)}.$ 

 Consider, $\Gamma=\pi_1(\Sigma_g)$ for $g\geq2$.  Note that $H^1(\Gamma, Ad\circ \rho)$ denotes the first cohomology group of $\Gamma$ with coefficients in $\mathfrak{g}$ (Lie algebra of $G$) twisted by $\Gamma \xrightarrow{\rho}G \xrightarrow{Ad}End(\mathfrak{g})$ where $Ad$ is the adjoint representation of $G$. Sikora proved that the tangent space of any element $[\rho]$ in the character variety of good representations of $Hom(\Gamma,G)$ is actually $H^1(\Gamma, Ad\circ \rho)$, that is,  $T_{[\rho]}\, {Hom^g(\Gamma,G)}/G=H^1(\Gamma, Ad \circ\rho)$ where $Hom^g(\Gamma,G)$ denotes the space of good representations of $Hom(\Gamma,G)$ and $G$ is any affine complex reductive group \cite[Section 14]{si}.

We already notice that, in the case of $G=\SL(n,\C)$, the irreducible representations are same as the good representations. Therefore, from the above paragraph when $G$ is $\SL(n,\C)$, it can be concluded that $$T_{[\rho]}\, {Hom^i(\Gamma,G)}/G=T_{[\rho]}\, {X^i_G(\Gamma)}=H^1(\Gamma, Ad\circ \rho) \hbox{ for all } [\rho] \in {X^i_G(\Gamma)}.$$ Simultaneously, from the second paragraph, it can be said that $T_{[\rho]} \,{X^{i,A}_G(\Gamma)}$, resp.   $T_{[\rho]}\, {X^{zd,A}_G(\Gamma)}$ $=$\,$H^1(\Gamma, Ad\circ \rho)$ for all $[\rho] \in {X^{i,A}_G(\Gamma)}$, resp. $X^{zd,A}_G(\Gamma)$. 

Suppose, $B:\mathfrak{g} \times \mathfrak{g} \to \C$ is a bilinear form. This is called  $Ad$-invariant when for all $x,y \in \mathfrak{g}$, $B(Ad(g)x,Ad(g)y)= B(x,y)$. 

Now, we have the following.

\begin{theorem}\label{2.4}
    Let $\Gamma=\pi_1(\Sigma_g) \,(g\geq 2)$, $G=\SL(n,\C).$ Then the character variety of $(Q^+, Q^-)$-Anosov irreducible representations, \hbox{resp.} $(P^+, P^-)$-Anosov Zariski dense  representations,  of $\Gamma$ into $G$ 
    denoted by $X^{i,A}_G(\Gamma)$, \hbox{resp.} $X^{zd,A}_G(\Gamma)$,  is a holomorphic symplectic manifold with respect to a holomorphic closed non-degenerate exterior 2-form $\omega_B.$ 
\end{theorem}
\begin{proof}
Let, $M$ be the complex manifold $X^{i,A}_G(\Gamma).$
Denote, $\Lambda^2\; T\,^*M$ = disjoint union of the vector spaces of  skew-symmetric bilinear forms $T_{[\rho]}\, M \, \times \, T_{[\rho]}\, M \to \C$ on $[\rho]\in M$. This is a vector bundle over  $M$ with a complex manifold structure.

Now,  for a symmetric, $Ad$-invariant, non-degenerate, bilinear form $B:\mathfrak{g} \times \mathfrak{g} \to \C$,
consider a map $\omega_B: M \to \Lambda^2\; T\,^*M $ such that on every point $[\rho] \in M$,  $\omega_B([\rho]) : {T_{[\rho]}} \,M \, \times\,  {T_{[\rho]}}\, M \to \C$ is 
  defined by   $ \omega_B(v,w)= B(v\,\cup\, w) \hbox{ for all } v,w \in T_{[\rho]}\, M=H^1(\Gamma, Ad\circ \rho)$,  where $\cup : {H^1(\Gamma, Ad \circ\rho)} \times {H^1(\Gamma, Ad \circ \rho)}\to H^2(\Gamma, {Ad\circ \rho} \otimes {Ad \circ \rho} )=\mathfrak{g} \otimes_{\Z\Gamma}\mathfrak{g} $ is  induced by the cup product (see \cite[Chapter\,5, Section\,3]{br}).
  
  From \cite[Corollary 59]{si},  it can be said that $w_B([\rho])$ is a skew-symmetric, non-degenerate bilinear form on $T_{[\rho]}\, M$ for every  $[\rho] \in M$.  Therefore $w_B([\rho])$ is a symplectic form on $T_{[\rho]}\, M$.  It follows from \cite[Section\,3.10]{go84}, \cite[Remark\,60]{si}  that,  $\omega_B$ is  a holomorphic map between two complex manifolds $M$ and $\Lambda^2\; T\,^*M$.   By \cite[Section\,1.7,\,1.8]{go84}, we get that $\omega_B$ is closed. Similarly, we can define such $\omega_B$ on the complex manifold $X^{zd,A}_G(\Gamma)$ with same properties since  $T_{[\rho]}\, {X^{zd,A}_G(\Gamma)}=H^1(\Gamma, Ad\circ \rho)$ for every $[\rho] \in {X^{zd,A}_G(\Gamma)}$. Therefore  $(X^{i,A}_G(\Gamma), \omega_B)$ and $(X^{zd,A}_G(\Gamma), \omega_B)$ are holomorphic symplectic manifolds with respect to the holomorphic closed non-degenerate exterior 2-form $\omega_B$.
\end{proof}

\end{document}